\documentclass[a4paper,twocolumn]{article}
\usepackage{bm}
\usepackage{amsfonts}
\usepackage{amsmath}
\usepackage{amssymb}
\usepackage{amsthm}
\usepackage{authblk}
\usepackage{color}

\begin{document}

\newtheorem{definition}{Definiton}[section]
\newtheorem{theorem}{Theorem}[section]
\newtheorem{lemma}{Lemma}[section]
\newtheorem{corollary}{Corollary}[section]

\title{Some Order-Theoretic Properties of the Zeros of the Zeta Function}
\author[1]{Boian Lazov\thanks{boian\_lazov@yahoo.com}}
\affil[1]{Department of Mathematics, University of Architecture, Civil Engineering and Geodesy, 1164 Sofia, Bulgaria}

\maketitle

\begin{abstract}
The (partially) ordered set of the non-trivial zeros of the zeta function with positive imaginary parts is considered. The order is the coordinatewise order inherited from $\mathbb{C}$. Some interesting properties regarding the minimal elements of this poset are proven.
\end{abstract}

\section{Introduction}
\
\indent To the best knowledge of the author the set of zeros of the zeta function has not been considered from an order-theoretic perspective. It seems such a viewpoint could prove useful, however, and this paper aims to draw the attention towards the goal of understanding the structure of said set using order theory.

\section{Definitions}\label{defs}
\
\indent We will consider the poset $(\mathbb{C},\le)$, where the order relation $\le$ is the coordinatewise order, defined by: $a+ib\le c+id$ if and only if $a\le c$ and $b\le d$. Let $Z$ denote the non-trivial roots of $\zeta(s)$ with non-negative imaginary parts, i. e.
\begin{align}
Z=\left\{ s\in\mathbb{C}:0\le s,\zeta(s)=0 \right\}.
\end{align}

\noindent Then we can consider $(Z,\le)$ with the inherited order relation from $(\mathbb{C},\le)$. Since the zeros of a meromorphic function are isolated and there are upper bounds on the number of zeros in a region of the critical strip (see e. g. \cite{ivic1984, borwein2008}) we can, for convenience, index the distinct imaginary parts of the elements of $Z$ and they will form a strictly increasing sequence $\left\{t_i\right\}_{i\in\mathbb{N}}$. If $t_1$ is the first member of this sequence (and thus the smallest one), the root $\sigma_1+it_1\in Z$ is known to lie on the critical line, i. e. $\sigma_1=\frac{1}{2}$, and is trivially a minimal element of $Z$.
\begin{definition}
Let
\begin{align}
Z_n:=\left\{ s\in Z:\Im(s)=t_n \right\}.\label{zndef}
\end{align}
Define the diameter of $Z_n$ by
\begin{align}
d(Z_n)=\underset{s_1,s_2\in Z_n}\max( \Re(s_1)-\Re(s_2) ).
\end{align}
\end{definition}

\indent Every $(Z_n,\le)$ is a totally ordered set. It has a least and a greatest element which we will denote by $\hat{\sigma}_n+i\hat{t}_n$ and $\tilde{\sigma}_n+i\tilde{t}_n$, respectively.

\section{Results}
\
\indent First we will state the following lemma:
\begin{lemma}
The Riemann hypothesis is true if and only if $(Z,\le)$ is a totally ordered set.
\end{lemma}
\begin{proof}
\textbf{1.)} Assume RH. $(Z,\le)$ is a partially ordered set by definition. Take $\sigma_n+it_n,\sigma_m+it_m\in Z$. Then $\sigma_n=\sigma_m=\frac{1}{2}$. Now we have $t_n\le t_m$ or $t_m\le t_n$ since $t_n,t_m\in\mathbb{R}$. Thus $\sigma_n+it_n\le\sigma_m+it_m$ or $\sigma_m+it_m\le\sigma_n+it_n$, i. e. $(Z,\le)$ is a totally ordered set.\\
\indent \textbf{2.)} Assume that $(Z,\le)$ is a totally ordered set. Let $\sigma_1+it_1,\sigma_k+it_k\in Z$ be such that $\sigma_1+it_1$ is the same as in section \ref{defs} and $\sigma_k\ne \frac{1}{2}$.\\
\indent If $\sigma_k<\frac{1}{2}$, then $\sigma_k<\sigma_1$ and $t_k>t_1$. Thus $\sigma_k+it_k||\sigma_1+it_1$, but $(Z,\le)$ is a totally ordered set which is a contradiction.\\
\indent If $\sigma_k>\frac{1}{2}$, then from the symmetry of the zeros about the $\Re(s)=\frac{1}{2}$ line there must exist $\sigma_l+it_l\in Z$ such that $\sigma_l=1-\sigma_k$ and $t_l=t_k$. We now have $\sigma_l<\sigma_1$ and $t_l>t_1$. Thus $\sigma_l+it_l||\sigma_1+it_1$ and again we reach a contradiction. This means that $\sigma_k=\frac{1}{2}$ and the proof is complete.
\end{proof}

\indent Next is a result concerning the number of minimal elements of $(Z,\le)$. Denoting the set of said elements by $\mathrm{Min}(Z,\le)$, we have:
\begin{theorem}\label{minth}
There is a bijection between $\mathrm{Mn}:=\left\{ Z_n:n>1,d(Z_n)>d(Z_i) \ \forall i \ 0<i<n \right\}$ and $\mathrm{Min}(Z,\le)\setminus \left\{ \sigma_1+it_1 \right\}$.
\end{theorem}
\begin{proof}
Let $f:\mathrm{Mn}\longrightarrow Z$ be such that $f(Z_n)=\hat{\sigma}_n+i\hat{t}_n$.\\
\indent \textbf{1.)} Consider some $Z_k\in \mathrm{Mn}$. Then $f(Z_k)=\hat{\sigma}_k+i\hat{t}_k$. From the symmetry of the zeros about the $\Re(s)=\frac{1}{2}$ line follows that
\begin{align}
\hat{\sigma}_k=\frac{1}{2}-\frac{d(Z_k)}{2}.\label{rpdist}
\end{align}

\noindent Take $\sigma_l+it_l\in Z$, such that $\sigma_l+it_l\le\hat{\sigma}_k+i\hat{t}_k$. Then
\begin{align}
t_l\le \hat{t}_k, \ \ \ \sigma_l\le\hat{\sigma}_k.\label{rpineq}
\end{align}
Also there exists some set $Z_l$ of the form (\ref{zndef}), such that $\sigma_l+it_l\in Z_l$.\\
\indent If $t_l=\hat{t}_k$, then $Z_l=Z_k$. Since $\hat{\sigma}_l+i\hat{t}_l$ is the least element in $Z_l$, $\hat{\sigma}_k=\hat{\sigma}_l\le\sigma_l$. Thus using (\ref{rpineq}) we get $\sigma_l=\hat{\sigma}_k$ and $\sigma_l+it_l=\hat{\sigma}_k+i\hat{t}_k$.\\
\indent If $t_l<\hat{t}_k$, then $d(Z_l)<d(Z_k)$, since $Z_k\in \mathrm{Mn}$. Similarly to (\ref{rpdist}), $\hat{\sigma}_l=\frac{1}{2}-\frac{d(Z_l)}{2}$. Then $\hat{\sigma}_k<\hat{\sigma}_l\le\sigma_l$, which is a contradiction with (\ref{rpineq}).\\
\indent Thus $\hat{\sigma}_k+i\hat{t}_k$ is a minimal element of $Z$, i. e. $f(Z_k)\in \mathrm{Min}(Z,\le)\setminus \left\{ \sigma_1+it_1 \right\}$.\\
\indent \textbf{2.)} Take $Z_k,Z_l\in \mathrm{Mn}$, such that $Z_k\ne Z_l$. Then $\hat{t}_k\ne \hat{t}_l$ and $f(Z_k)\ne f(Z_l)$. Thus $f$ is injective.\\
\indent \textbf{3.)} Consider some $\sigma_k+it_k\in \mathrm{Min}(Z,\le)\setminus \left\{ \sigma_1+it_1 \right\}$. There exists some $Z_k$, such that $\sigma_k+it_k\in Z_k$. Then $\hat{\sigma}_k+i\hat{t}_k\le\sigma_k+it_k$, but since $\sigma_k+it_k$ is a minimal element, we have
\begin{align}
\hat{\sigma}_k+i\hat{t}_k=\sigma_k+it_k.\label{minmap}
\end{align}
Now suppose that for some $l<k$ there exists $Z_l$, such that $d(Z_l)\ge d(Z_k)$. Then $\hat{t}_l<\hat{t}_k$ and from (\ref{rpdist}) $\hat{\sigma}_l\le \hat{\sigma}_k$. Thus
\begin{align}
\hat{\sigma}_l+i\hat{t}_l<\hat{\sigma}_k+i\hat{t}_k=\sigma_k+it_k,
\end{align}

\noindent but $\sigma_k+it_k$ is minimal and we reach a contradiction. It follows that $d(Z_l)<d(Z_k)$, i. e. $Z_k\in \mathrm{Mn}$. From (\ref{minmap}) $f(Z_k)=\sigma_k+it_k$, so $f$ is surjective. This completes the proof.
\end{proof}

\indent A dual result can be stated for the number of maximal elements of $(Z,\le)$:
\begin{theorem}\label{maxth}
There is a bijection between $\mathrm{Mx}:=\left\{ Z_n:d(Z_n)>d(Z_i) \ \forall i \ i>n \right\}$ and $\mathrm{Max}(Z,\le)$.
\end{theorem}

\indent The bijection here is given by $g:\mathrm{Mx}\longrightarrow Z$, such that $g(Z_n)=\tilde{\sigma}_n+i\tilde{t}_n$. After that the proof is analogous to the previous one, considering that
\begin{align}
\tilde{\sigma}_n=\frac{1}{2}+\frac{d(Z_n)}{2}.
\end{align}
\begin{corollary}\label{mincol}
RH is not true if and only if $(Z,\le)$ has at least $2$ minimal elements.
\end{corollary}
\begin{proof}
RH not true implies that there is a set $Z_n$ with $d(Z_n)> 0$, i. e. $Z_n\in \mathrm{Mn}$, and by theorem \ref{minth} $(Z,\le)$ has a minimal element distinct from $\sigma_1+it_1$.\\
\indent Conversely, if $(Z,\le)$ has at least $2$ minimal elements, then there exists a set $Z_n\in\mathrm{Mn}$ with $d(Z_n)>0$, again by theorem \ref{minth}, and RH is false.
\end{proof}

\indent Similarly:
\begin{corollary}\label{maxcol}
RH is not true if $(Z,\le)$ has at least $1$ maximal element.
\end{corollary}

\indent This follows from theorem \ref{maxth} analogously to the previous corollary.\\
\indent We should note here that corollary \ref{maxcol} doesn't include an "only if" statement. This is the case, because for example the sequence $\{ d(Z_i) \}_{i\in\mathbb{N}}$ can be monotonically increasing and have a strictly increasing subsequence, in which case $\mathrm{Mx}$ is empty and so is $\mathrm{Max}(Z,\le)$. Then RH is not true, but there are $0$ maximal elements.


\begin{thebibliography}{99}

\bibitem{ivic1984} "A zero-density theorem for the Riemann zeta-function", A. Ivic, Trudy Mat. Inst. Steklov. 163, 85-89 (1984)

\bibitem{borwein2008} "The Riemann Hypothesis", P. Borwein, S. Choi, B. Rooney, A. Weirathmueller, Springer-Verlag New York (2008)

\end{thebibliography}
\end{document}